\documentclass{amsart}

\usepackage{amssymb,amsfonts,amsmath,amsthm}

\PassOptionsToPackage{hyphens}{url}	
\usepackage{hyperref}
\usepackage{breakurl}  

\usepackage{epsfig}
\usepackage{longtable,verbatim}
\usepackage{dcolumn}
\usepackage{comment}
\newcolumntype{2}{D{.}{}{2.0}}

\newtheorem{theorem}{Theorem}[section]{\bfseries}{\itshape}
\newtheorem*{theorem*}{Theorem}{\bfseries}{\itshape}
{\bfseries}{\itshape}
\newtheorem{Lem}[theorem]{Lemma}{\bfseries}{\itshape}
{\bfseries}{\itshape}
{\bfseries}{\itshape}
{\bfseries}{\itshape}
{\bfseries}{\itshape}

\theoremstyle{remark}
\newtheorem{remark}[theorem]{Remark}{\bfseries}{\itshape}
\newtheorem{example}[theorem]{Example}{\bfseries}{\itshape}

\DeclareMathOperator{\lcm}{lcm}

\usepackage[vlined,ruled,linesnumbered]{algorithm2e}
\usepackage{tikz}
\usepackage{graphicx}
 \usepackage[foot]{amsaddr}
 
 \makeatletter
\renewcommand{\email}[2][]{%
  \ifx\emails\@empty\relax\else{\g@addto@macro\emails{,\space}}\fi%
  \@ifnotempty{#1}{\g@addto@macro\emails{\textrm{(#1)}\space}}%
  \g@addto@macro\emails{#2}%
}
\makeatother

\newcommand{\algMn}{1}		
\newcommand{\algdeltan}{2}	
\newcommand{\algdeltanT}{3}	
\newcommand{\multipliers}{4}	
\newcommand{\fastalg}{5}	
\newcommand{\bachmod}{6}	

\newcommand{\E}{{\mathbb E}}	
\newcommand{\V}{{\mathbb V}}	

\newtheorem*{remark*}{Remark}

\title{Algorithms for the Multiplication Table Problem}

\author{ Richard Brent}
\address[Brent]{ Australian National University, Canberra, ACT 2600, Australia}

\author{ Carl Pomerance}
\address[Pomerance]{ Mathematics Department, Dartmouth College, Hanover, NH 03755}

\author{ David Purdum}
\author{ Jonathan Webster}
\address[Purdum, Webster]{Butler University, Indianapolis, IN 46208}


\begin{document}

\begin{abstract}
Let $M(n)$ denote the number of distinct entries in the $n \times n$
multi\-plication table. The function $M(n)$ has been studied by
Erd\H{o}s, Tenenbaum, Ford, and others,
but the asymptotic behaviour of $M(n)$ as $n \to \infty$
is not known precisely. Thus, there is some interest in algorithms for
computing $M(n)$ either exactly or approximately.
We compare several algorithms for computing
$M(n)$ exactly, and give a new algorithm that has a subquadratic running 
time. We also present two Monte Carlo algorithms 
for approximate computation of $M(n)$.
We give the results of exact computations
for values of $n$ up to $2^{30}$, 
and of Monte Carlo computations for $n$
up to  $2^{100,000,000}$,
and compare our experimental results with Ford's
order-of-magnitude result.
\end{abstract}

\maketitle

\section{Introduction}

Although a multiplication table is understood by a
typical student in elementary school, there remains much that we do not know
about such tables.
In 1955, Erd\H{o}s studied the problem of counting the number $M(n)$ 
of distinct
products in an $n\times n$ multiplication table.  
That is, $M(n) := |\{ i j: 1 \leq i,j \leq n \}|$.  
In~\cite{erdos55}, Erd\H{o}s showed that $M(n) = o(n^2)$.  Five years later,
in~\cite{erdos60}, he obtained
\begin{equation}
M(n) = \frac{n^2}{(\log n)^{c + o(1)} } \mbox{ as } n \rightarrow \infty,
	\label{eq:Erdos-improved}
\end{equation}
where (here and below) 
$c  =  1 - (1 + \log \log 2)/\log 2 \approx 0.086071$.  
In 2008, Ford 
\cite{ford08,ford08_simpler} gave the correct order of magnitude
\begin{equation}
M(n) =\Theta(n^2/\Phi(n)),		\label{eq:FordM}
\end{equation}
where
\begin{equation}
\Phi(n) := (\log n)^c (\log\log n)^{3/2}	\label{eq:Phi}
\end{equation}
is a slowly-growing function.\footnote{In~\eqref{eq:FordM}, the notation
$f =\Theta(g)$ means $f =O( g)$ {and} $g =O( f)$,
i.e.\ there are positive constants $A$ and $B$ such that
$Ag(n) \le f(n) \le Bg(n)$ for all sufficiently large~$n$.}

Note that \eqref{eq:FordM} is not a true asymptotic formula, 
as $M(n)/(n^2/\Phi(n))$ might or might not tend to a limit as
$n\to\infty$.  The computation of
$M(n)$ for various large $n$ could suggest (though not prove)
the true behaviour of $M(n)$ as $n \to \infty$.
The history of such computations goes back to Brent and Kung \cite{b_k},
who considered how much area and time are needed to perform an $n$-bit binary
multiplication on a VLSI chip.
For this, they needed a lower bound on $M(2^n-1)$.
In 1981 they computed%
\footnote{ Brent and Kung actually computed $M(2^n-1)+1 =|\{ i j: 0 \leq
i,j < 2^n \}|.$ For consistency in the exposition, we translate their
results to the definition of $M(n)$ stated above.} 
\hbox{$M(2^n -1)$} for $1 \leq n \leq 17$. 
In unpublished work dating from 2012, 
the first two authors revisited the problem, extending the
computation through $n=25$, and exploring Monte Carlo estimates
for larger $n$.
Some years later, the fourth author discovered several new algorithms for
the exact computation of $M(n)$. In this paper we present both exact
and Monte Carlo algorithms.

It is  useful to distinguish between an
algorithm for \emph{evaluating} $M(n)$ at one given integer $n$, and an 
algorithm for \emph{tabulating} $M(k)$ for all integers $k$ in the
interval $[1,n]$, which we may refer to simply as ``tabulating $M(n)$''.
Several of the exact evaluation algorithms can be modified
in an obvious way to give a tabulation algorithm with essentially the same
time and space complexity.\footnote{Space is measured in 
bits, and does not include any space required to store the output.}
This is not true of the Monte Carlo algorithms,
which output an estimate of a single value $M(n)$ and give no information
on $M(k)$ for $k \ne n$.

Regarding exact algorithms,
our contributions are an extension of the previous numerical work on
evaluation/tabulation of $M(n)$,
and the development of an asymptotically faster (subquadratic) tabulation
algorithm.
Specifically, we can evaluate $M(k)$ for all $k \le n$ in time
$O (n^2/L^{1/\sqrt{2} + o(1)} ),$ 
where  
\[L = L(n) := \exp{\left( \sqrt{ \log n \log \log n }\right) },\]
using $O(n)$ space. We present this algorithm
(Algorithm~$\fastalg$) via a series of steps.   
First, we explain a straightforward algorithm (Algorithm~$\algMn$)
to evaluate $M(n)$ in time and space $O(n^2)$.
This algorithm\footnote{The space requirement
can be reduced by segmentation, see \S\ref{subsec:sieve}.}
was used by Brent and Kung~\cite{b_k}.
Second, we show (Algorithm~$\algdeltan$) that 
we can evaluate $M(n)$ given $M(n-1)$, 
using only $O(n)$ space and $O(n\log n)$ time.\footnote{The
time bound can be improved, see
Theorems~\ref{thm:Alg2bounds}--\ref{thm:fastalg}
in \S\ref{subsec:differences}.}
This incremental approach 
naturally leads to a tabulation algorithm
(Algorithm~$\algdeltanT$) that uses $O(n)$ space
and $O(n^2\log n)$ time.
Finally, we refine the incremental approach to obtain a subquadratic
tabulation algorithm (Algorithm~$\fastalg$): see Theorem~\ref{fast_comp}
for the time bound.

For arguments $n \le 2^{30}$, our implementation
of Algorithm~$\fastalg$
is slower than that of the best quadratic algorithm. This is a
familiar phenomenon: for many other problems (e.g. multiplication of
$n$-bit integers, or of $n \times n$ matrices) the asymptotically fastest
known algorithm is not necessarily the fastest in practice.

We have tabulated $M(n)$ for $n \le 2^{30}$, using
a variant of Algorithm~$\algdeltanT$.
For confirmation of the numerical results we used
Algorithm~$\algMn$ (with segmentation and parallelisation)
for various values of $n$,  
including $n = 2^{k}-1$ for $k = 1, 2,\ldots, 30$.

The known exact (quadratic or subquadratic) evaluation/tabulation
algorithms are too slow to go much past $n = 2^{30}$,
so for larger $n$
it is necessary to resort to approximate (Monte Carlo) methods.
We give two Monte Carlo algorithms,
which we call \emph{Bernoulli} and \emph{product},
for reasons that will be evident later (see \S\ref{approx_eval}).
In each case, we avoid the problem of factoring large integers by using
Bach's algorithm~\cite{bach_factor,kalai_factor} 
for generating random integers in factored form. 
As far as we are aware, this project represents the first time that the Bach algorithm
for producing random factored numbers has been usefully implemented.
The speed of the Monte Carlo algorithms
depends mainly on the time required for testing the primality%
\footnote{Or ``probable'' primality, 
since a small probability of error is acceptable in the context of
a Monte Carlo computation, see \S\ref{subsec:Bach}.}
of large integers, which can be done much faster than factoring integers
of comparable size \cite{AKS,Rabin80}.

The paper is organized as follows.
Section $\ref{exact_eval}$ is concerned with exact algorithms for
evaluating/tabulating $M(n)$.
After a brief overview of the sieve of Eratosthenes as a precursor to
various ways of evaluating $M(n)$, we start with the method used 
by Brent and Kung~\cite{b_k}.
We then show how to tabulate $M(n)$ 
in time $O(n^2\log n)$ using an incremental algorithm.
In fact, the time bound can be reduced slightly by using a result
of Ford~\cite{ford08}, as we show in 
Theorems~\ref{thm:Alg2bounds}--\ref{thm:fastalg} and Remark~\ref{rem:Mnlogn}.  
In \S\ref{modcomp} we consider generating products in a multiplication
in specific residue classes.
We then show (in Theorem~\ref{fast_comp}) that the incremental algorithm 
can be  modified to tabulate $M(n)$ in time $O (n^2/L^{1/\sqrt{2} + o(1)} )$.
We remark that $\log n = L^{o(1)}$, so $\log n$ factors can be subsumed into 
the $o(1)$ term in the exponent of~$L$.

Section \ref{approx_eval} describes and compares 
our two Monte Carlo algorithms for 
estimating $M(n)$. We consider the variance
in their estimates for a given number of independent random trials.
Lemma~\ref{lemma:compare_variance} shows that, considering only the
variance, the product algorithm is more accurate than the Bernoulli 
algorithm. This does not necessarily mean that it is preferable in
practice, as factors such as the time per trial and space requirements
need to be considered.

Finally, Section~\ref{sec:numerical}
contains numerical results (exact up to $n = 2^{30}-1$,
and approx\-imate up to $n = 2^{100,000,000}$), 
comments on implementations of the
algorithms, and some conclusions regarding the asymptotic behaviour
of~$M(n)$.

\begin{remark}
In~\eqref{eq:Phi}, the factor $(\log n)^c$ is asymptotically
larger than the factor $(\log\log n)^{3/2}$. However, for small~$n$, the
second factor varies more rapidly than the first. 
Write $x := \log n$, $A = A(x) := x^c$, $B = B(x) := (\log x)^{3/2}$. 
Thus $\Phi(n) = AB$ and, taking logarithmic derivatives,
we have $\Phi'/\Phi = A'/A + B'/B$.
Now $|A'/A| < |B'/B|$ if $c/x < 3/(2x\log x)$, 
i.e.\ if $x < \exp(3/2c) \approx 3.7\times 10^7$, or
\[
n < \exp(\exp(3/2c)) \approx 2^{53,431,908}.
\]
Consequently, our numerical results
up to $n = 2^{100,000,000}$ barely extend to the region where the
true asymptotic behaviour of $M(n)$ becomes evident.
\end{remark}

\section{Exact Evaluation of $M(n)$}	\label{exact_eval}

Our model of computation is a standard random access machine with infinite,
direct-access memory.  Memory may be addressed at the bit-level or at the
word level, and the word size is $\Theta( \log n)$ bits if $n$ is the input
size.  We assume that arithmetic operations, memory access, and other basic
operations take unit time.  We count space in bits and do not include
any space used to store the output.

\subsection{Sieve of Eratosthenes}	\label{subsec:sieve}

The algorithms for 
evaluating $M(n)$ resemble the sieve of
Eratosthenes, 
the  simplest implementation of which involves,
for each $1<k\le n^{1/2}$, removing
the multiples of $k$ from $(k,n]$.  
This naive implementation uses $O(n \log n)$ time and $O(n)$
space and finds all primes up to~$n$.  There is a large body of
literature, both practical and theoretical, dealing with improvements and
variations to this sieve.
We refer to Helfgott \cite{helfgott17} for a
summary of the literature. Here, we highlight the aspects that are relevant 
for computing $M(n)$.
In practice, we may be limited by a space constraint; lowering the space
used by an algorithm 
may turn otherwise infeasible computations into feasible ones. 
For example, 
before marking off multiples of $k$ in $(k,n]$, we may segment this interval
into subintervals.   The asymptotic run-time remains unchanged so long as the
``marking off" process is not doing ``empty work'',
i.e.\ so long as $(n-k)/k$ is not small. Using this idea, the space
bound may be reduced to $O(n^{1/2})$ with straightforward segmentation 
of the interval $[1,n]$. Helfgott~\cite{helfgott17} reduces the space
requirement further by using Diophantine approximation to predict
which integers less than $n^{1/2}$ have multiples in a given subinterval.  
This prediction process allows
sieving on intervals of size $O(n^{1/3}(\log n)^{5/3})$ 
at no asymptotic cost in time \cite[Main Theorem]{helfgott17}.

\subsection{Computing $M(n)$ Directly}		\label{subsec:directly}

We can explicitly construct each product in a multiplication table and count
the number of distinct products using Algorithm~$\algMn$.  The algorithm
exploits the symmetry of the multiplication table.

\SetAlgoRefName{\algMn}
\begin{algorithm}
\caption{Computing $M(n)$ directly}
\SetKwInOut{Input}{Input} \SetKwInOut{Output}{Output}

\Input{An integer $n$}
\Output{$M(n)$}
Initialize a bit vector $A$ of length $n^2$ to $0$.\\ 
\For{ $ 1 \leq  i \leq n $ }{
  \For{ $i \leq j \leq n $ }{
  Set $A[ij] = 1$
  }
 } 
 \Return Hamming weight of $A$
\end{algorithm}

The following lemma is obvious from counting the number of times the body
of the inner loop is executed .  We note that the area associated with
the table is $n^2$.
\begin{Lem}			\label{lem:naive}
Algorithm $\algMn$ computes $M(n)$ in time $O(n^2)$ and space $O(n^2)$.
\end{Lem}

Algorithm $\algMn$ looks similar to the sieve of Eratosthenes 
(for finding the primes smaller than $n^2$), and
many of the tricks that are known to speed up the latter may also be used to
speed up Algorithm~$\algMn$.
The key difference is the stopping point for marking off multiples of
$i$; Algorithm~$\algMn$ only marks off through the $n$th multiple of $i$. 
Because of this early stopping point, Algorithm~$\algMn$ has time bound
$O(n^2)$, whereas the corresponding bound for a naive version of the sieve
of Eratosthenes is $O(n^2\log n)$.

The space bound of Lemma~\ref{lem:naive} can be reduced by modifying
the algorithm.
As discussed above, standard segmenting allows subintervals of
size $O((n^2)^{1/2})=O(n)$. By using Diophantine approximation 
the space bound could even be
reduced to $O(n^{2/3}(\log n)^{5/3})$.

Suppose that it is possible to store all $n^2$ bits of the vector $A$.
If the bit vector $A$ associated 
with the computation of $M(n-1)$ is saved, 
then $M(n)$ may be computed in $O(n)$ additional time.  
We simply count how many bits are \emph{not} set 
in $S = [A[n], A[2n], \ldots, A[n^2]]$, 
and increment $M(n-1)$ by that amount.  Let the number of set bits
in $S$ be denoted by $\delta(n)$, so the number of unset bits 
is $n-\delta(n)$. 
We can compute $\delta(n)$ in $O(n)$ time,
and we can compute $M(n)$ using
\begin{equation}			\label{eq:Mviadelta}
 M(n) = M(n-1) + (n - \delta(n))  = 
  \sum_{k = 1}^{n}\big( k - \delta(k)\big) = \frac{n^2 + n}{2} - 
  \sum_{k = 1}^n \delta(k).
\end{equation}
\S\ref{subsec:differences} shows how to compute $\delta(n)$ almost as quickly, 
using only $O(n)$ space.  

\subsection{Computing $M(n)$ Incrementally}	\label{subsec:differences}

We compute $M(n)$ incrementally using~\eqref{eq:Mviadelta}.
More precisely, we compute
$\delta(k)$ for all $k \leq n$ where $\delta(n)$ counts the elements in $\{
n, 2n, 3n, \ldots, n^2 \} = \{ mn : 1 \leq m \leq n\}$ that appear in the
$(n-1) \times (n-1)$ multiplication table.  If $mn$ appears in the smaller
multiplication table then it may be factored so that each factor is strictly
less than $n$.  If $m = ij$ and $n = gh$, then $mn = (ij)(gh) = (ih)(jg)$.
If $ih < n$ and $jg <n$ then the product $mn$ has already appeared in the table.
Observe that $ih < n$ iff $i < g$ and $jg <n$ iff $j<h$.  
Thus, to compute $\delta(n)$, we need to count distinct products
$ij$ with $0< i < g$ and $0 < j < n/g$ for each divisor $g$ of $n$ with
$g \leq \sqrt{n}$.
By counting the distinct products in
the shape formed by the union of rectangles 
whose boundaries are determined by the divisors
of $n$, we may compute $\delta(n)$.

\SetAlgoRefName{\algdeltan}
\begin{algorithm}
\caption{Computing $\delta(n)$}
\SetKwInOut{Input}{Input} \SetKwInOut{Output}{Output}

\Input{$D = [[d_0=1,n],\ldots, [d_{\ell -1}, n/d_{\ell -1}]]$,
containing the ordered pairs of divisors of $n$,
where $d_{\ell -1}$ is the largest divisor in $[1, \sqrt{n}]$.}
\Output{$\delta(n)$}
\BlankLine
Initialize counters $i = 1$ and $k =0$\\
Initialize a bit vector $A$ of length $n$ to 0.\\ 
\For{ $ i < D[\ell -1][0]$ }{
\If{ $i == D[k][0]$ }{ Increment $k$}
  \For{ $i \leq j < D[k][1]$ }{
  Set $A[ij] = 1$
  }
 } 
 \Return Hamming weight of $A$
\end{algorithm}

\begin{remark} 
If the input to Algorithm $\algdeltan$ is missing one or more divisor pairs, 
then the output (Hamming weight of $A$) is a lower bound on $\delta(n)$.

\end{remark}

\begin{example}
In Figure $\ref{example_42}$, the gray area
corresponds to the products that Algorithm~$\algdeltan$ constructs
given the divisor pairs $2\cdot 21$, $3 \cdot 14$, and $6 \cdot 7$ of $42$.

\begin{figure}[ht]
\caption{ The shape for computing $\delta(42)$}
\label{example_42}
    \centering
    \resizebox{0.8\textwidth}{!}{

\begin{tikzpicture}

\fill[gray] (1,1) -- (21,1) -- (21,2) -- (14,2) -- (14,3) -- (7, 3) -- (7, 6) -- ( 5, 6) -- (5, 5) -- (4, 5) -- (4, 4 ) -- (3, 4) -- (3, 3) -- (2, 3) -- (2, 2 ) -- (1, 2) --   cycle;

\draw (0, 0) grid (22, 7);
\draw [ultra thick] (1,0) -- (1,7);
\draw [ultra thick] (0,1) -- (22,1);

\foreach \j in {1,...,6} {
	\foreach \i in {1,...,21} {
        		\pgfmathtruncatemacro{\label}{\i*\j};
        		\node (\label) at (\i+ .5, \j +.5) {\label};
}}

\foreach \j in {1,...,6} {
	\foreach \i in {1,...,1} {
        		\pgfmathtruncatemacro{\label}{\i*\j};
        		\node (\label) at (.5, \j +.5) {\label};
}}

\foreach \j in {1,...,1} {
	\foreach \i in {1,...,21} {
        		\pgfmathtruncatemacro{\label}{\i*\j};
        		\node (\label) at (\i + .5, .5) {\label};
}}

\node (x) at (.5,.5) {\large $\times$};

\end{tikzpicture} 

 }
\end{figure}
\end{example}

Algorithm $\algdeltan$ runs in time proportional to the 
area of the shaded region (which we call \emph{the shape}).
In general, an upper bound is $O(n \log n)$, obtained by noting that
no product in a rectangle is larger than $n$, so the total area is
bounded by the area under a hyperbola.  
A different upper bound is $O(n \tau(n))$, where $\tau(n)$ counts the
divisors of $n$.  This bound comes from the fact that, 
for each divisor of $n$, we
construct a rectangle of area less than $n$.  Both of these upper bounds
can over-estimate. 
The $O(n\log n)$ bound over-estimates when $n$ is not smooth,
and the $O(n \tau(n))$ bound
over-estimates when $n$ is smooth~\cite{CP}. 
The first bound may be used to show that
$M(n)$ can be tabulated in time $O(n^2 \log n)$; 
this bound can be reduced to $o(n^2\log n)$ by using deeper results on
divisors, as in Theorem~\ref{thm:fastalg}.

Following \cite{erdos-tenenbaum,ford08}, we
let $\tau(n; y, z)$ be the number of divisors $d$ of $n$ which satisfy 
$y <d \leq z$, and 
$\tau^{+}(n) = | \{ k \in \mathbb{Z} : \tau(n, 2^k, 2^{k+1})\geq 1 \}|$. 
Lemma~\ref{lem:tauplus} (due to Ford) 
bounds the mean value of $\tau^{+}(n)$.

\begin{Lem}[{Ford \cite[Corollary 5]{ford08}}]		\label{lem:tauplus}
If $c = 0.086\ldots$ is as above, then
\[
\frac{1}{n} \sum_{k \leq n}\tau^{+}(k)
=\Theta\Big(\frac{\log n }{\Phi(n)}\Big)\,.
\]
\end{Lem}

\begin{theorem}		\label{thm:Alg2bounds}
Algorithm $\algdeltan$ computes $\delta(n)$ in space $O(n)$ and
in time $O(n \tau^{+}(n))$.
\end{theorem}

\begin{proof}
By the above discussion concerning $\delta(n)$, the algorithm is correct.  
As $i$ increases, the counter $k$ keeps track of which rectangle boundary 
to use.  
The counter $j$ is then bounded above by the appropriate divisor of $n$.

The space bound is obvious, since the vector $A$ uses $n$ bits.

For the time bound, recall that
the run-time is proportional to an area, say $\mathcal{A}$.   
For each $k$, consider all the divisors of $n$ in the interval $(2^k,
2^{k+1}]$.  They all have the same bottom left corner, namely, the origin,
and shapes range from $2^k \times n/2^k$ to $2^{k+1} \times n/2^{k+1}$. 
Hence they are all enclosed by a rectangle of shape $2^{k+1} \times n/2^k$
which has area $2n$.  Thus we get an upper bound $\mathcal{A} \le 2n
\tau^{+}(n)$.
\end{proof}

Clearly Algorithm~$\algdeltan$ can be invoked repeatedly to tabulate
$M(n)$. For reference we call this (tabulation algorithm)
\emph{Algorithm~$\algdeltanT$}.

\begin{theorem}				\label{thm:fastalg}
Algorithm $\algdeltanT$ tabulates $M(n)$ in space $O(n)$ and time 
\[ O\left( \frac{n^2 \log n}{\Phi(n)} \right). \]
\end{theorem}

\begin{proof}
We compute $M(n)$ by evaluating $\delta(k)$ for $1 \leq k \leq n$.
Using Theorem~\ref{thm:Alg2bounds},
the total run-time is
\[
O\left( \sum_{k \leq n} k \tau^{+}(k) \right)
=O\left( n\sum_{k \leq n}  \tau^{+}(k) \right),
\]
so the result now follows from Lemma~\ref{lem:tauplus}.
\end{proof}

\begin{remark}				\label{rem:Mnlogn}
In view of Ford's result~\eqref{eq:FordM}, the time bound given in
Theorem~\ref{thm:fastalg} can be written as
$O(M(n)\log n)$. We do not know how to give a direct proof of
this without using Ford's results.
\end{remark}

The space bound in Theorem~$\ref{thm:Alg2bounds}$ is for a naive 
implementation. It is not difficult
to see that it  can be reduced to $O(n^{1/2})$ with straightforward
segmentation, or even to $O(n^{1/3}(\log n)^{5/3})$ via Diophantine 
approximation, as in \cite{helfgott17}. 
Algorithm~$\algdeltanT$ represents an improvement by a
factor of $n$ in the naive storage cost and a significant improvement in
run-time for the tabulation problem.  If only a single evaluation is
required, then Algorithm~$\algMn$ may be faster.  
In practice, we observed that 
Algorithm~$\algdeltanT$ is competitive with Algorithm~$\algMn$.
This may be due to
different implied constants in the big-$O$ bounds,
and because Algorithm~$\algMn$ has
a larger memory requirement, which can cause a deviation from the
expected quadratic run time due to cache effects~\cite[Chapter~2]{HP}.
In the next subsection we explain how generating products in specific
residue classes can be used to speed up Algorithm~$\algdeltan$.

\subsection{Working ``modulo $w$''} \label{modcomp}

We may generate products in a multiplication table in specific residue
classes; this is akin to sieving with a wheel, and has two advantages. 
First, if $w$ is the modulus, then the vector used in
Algorithm $\algdeltan$ may be declared to be of
size $\lfloor n/w \rfloor$ and unique products may be counted by residue class. 
Second, by not explicitly constructing small consecutive products, 
but simply counting them, we get a faster algorithm.  
In the following we illustrate these ideas with the examples
$w=1, 2$, and $6$.

\subsubsection{Working ``modulo $1$''}

If $n$ is not prime, then the first row of the table contains the
consecutive integers less than the largest nontrivial divisor of $n$.  Store
the number of consecutive integers and initialize the bit vector $A$ so that
the zero index is associated with the largest divisor of $n$.  Iterate
through each row of the multiplication table starting at the first entry
greater than or equal to the largest divisor.  
Figure $\ref{example_42_frf}$ shows
the area that is considered in computing $\delta(42)$.  The light
gray products are all accounted for because the first row has $20$ distinct
products.  We only construct the products greater than $20$, which are shown
in dark gray.  This improvement reduces both the time and space
requirements by a factor of $(1-1/p_1)$, where $p_1$ is the smallest prime
factor of~$n$.  As a consequence, it is easy to see that $\delta(2p) = p-1$
if (as usual) $p$ is a prime.

\begin{figure}[ht]
\caption{  The shape for computing $\delta(42)$ working modulo $1$.}
\label{example_42_frf}
    \centering
    \resizebox{0.9\textwidth}{!}{

\begin{tikzpicture}

\fill[lightgray] (1,1) -- (21,1) -- (21,2) -- (14,2) -- (14,3) -- (7, 3) -- (7, 6) -- ( 5, 6) -- (5, 5) -- (4, 5) -- (4, 4 ) -- (3, 4) -- (3, 3) -- (2, 3) -- (2, 2 ) -- (1, 2) --   cycle; 
\fill[gray] (11,2) -- (14, 2) -- (14, 3) -- (11, 3) -- cycle;
\fill[gray] (6,4) -- (7, 4) -- (7, 6) -- (5, 6) -- (5, 5) -- (6, 5) -- cycle;

\draw (0, 0) grid (22, 7);
\draw [ultra thick] (1,0) -- (1,7);
\draw [ultra thick] (0,1) -- (22,1);

\foreach \j in {1,...,6} {
	\foreach \i in {1,...,21} {
        		\pgfmathtruncatemacro{\label}{\i*\j};
        		\node (\label) at (\i+ .5, \j +.5) {\label};
}}

\foreach \j in {1,...,6} {
	\foreach \i in {1,...,1} {
        		\pgfmathtruncatemacro{\label}{\i*\j};
        		\node (\label) at (.5, \j +.5) {\label};
}}

\foreach \j in {1,...,1} {
	\foreach \i in {1,...,21} {
        		\pgfmathtruncatemacro{\label}{\i*\j};
        		\node (\label) at (\i + .5, .5) {\label};
}}

\node (x) at (.5,.5) {\large $\times$};

\end{tikzpicture} 

 }
\end{figure}

\subsubsection{Working ``modulo $2$''}

If $n$ is composite and not of the form $2p$ for $p$ a prime, then its shape
has nontrivial entries in the first two rows.  Create a bit vector
associated with odd numbers.  The first row is indexed by
consecutive odd numbers
up to some bound.  Either the first row or the second row will contain the
bound for the consecutive even numbers that are stored.  For rows associated
with an odd multiplier, start with the lower bound associated with the odd
vector and iterate through the table creating only the odd entries.  For the
even vector, consider even rows and the even numbers in the odd rows.  This
reduces the time by reducing area although the overhead in setting up the
loops to iterate through the table in the specified manner is higher.  More
importantly, it reduces the memory requirement.  By splitting the products
into residue classes modulo $2$, we require half the storage.  The above
discussion also makes it easy to see that $\delta(3p) = p-1 + \lfloor
(p-1)/2 \rfloor$.

Figure $\ref{example_75_f2rf}$ shows the area that is considered in
computing $\delta(75)$ modulo 2.  The products in light gray are accounted
for by a counting argument and the products in dark gray are constructed. 
That is, the bit vector storing even numbers starts at $50$, and then the
even products $52$ and $56$ in dark gray are constructed.  Thus, there are
$24 + 2 = 26$ unique even products.  The bit vector storing odd numbers
starts at $25$, and constructs the products $27, 33, 39$.  Therefore, there
are $12 + 3 = 15$ unique odd products, and $\delta(75) = 26 + 15 = 41$.

\begin{figure}[ht]
\caption{  The shape for computing $\delta(75)$ working modulo $2$. }
\label{example_75_f2rf}
    \centering
    \resizebox{0.9\textwidth}{!}{

\begin{tikzpicture}

\fill[lightgray] (1,1) -- (25,1) -- (25,3)  --  (15,3) -- (15, 5) -- (4, 5) -- (4,4) -- (3, 4) -- (3, 3) -- (2, 3) -- (2,2) -- (1, 2) --  cycle; 
\fill[gray] (9,3) -- (9, 4) -- (10, 4) -- (10,3) -- cycle;
\fill[gray] (11,3) -- (11, 4) -- (12, 4) -- (12,3) -- cycle;
\fill[gray] (13,3) -- (14,3) -- (14, 4) -- (15, 4) -- (15,5) -- (13, 5) -- cycle ;

\draw (0, 0) grid (26, 6);
\draw [ultra thick] (1,0) -- (1,6);
\draw [ultra thick] (0,1) -- (26,1);

\foreach \j in {1,...,5} {
	\foreach \i in {1,...,25} {
        		\pgfmathtruncatemacro{\label}{\i*\j};
        		\node (\label) at (\i+ .5, \j +.5) {\label};
}}

\foreach \j in {1,...,5} {
	\foreach \i in {1,...,1} {
        		\pgfmathtruncatemacro{\label}{\i*\j};
        		\node (\label) at (.5, \j +.5) {\label};
}}

\foreach \j in {1,...,1} {
	\foreach \i in {1,...,25} {
        		\pgfmathtruncatemacro{\label}{\i*\j};
        		\node (\label) at (\i + .5, .5) {\label};
}}

\node (x) at (.5,.5) {\large $\times$};

\end{tikzpicture} 

 }
\end{figure}

\subsubsection{Working ``modulo $6$''}

A naive invocation of Algorithm $\algdeltan$ to compute $\delta(377)$
requires the construction of $270$ products.  By constructing products in
residue classes modulo $6$ only $119$ products need to be constructed.  In
Figure $\ref{example_377}$, we see that the sixth row tells us there are $28$
consecutive multiples of $6$.  Therefore, we only need to construct products
$0 \pmod{6}$ that are greater than $168$.  Similarly, the third row tells us
that there are $14$ consecutive numbers $3 \pmod{6}$.  Therefore, we only
construct products $3 \pmod{6}$ that are greater than $84$.  The second row
tells us that we only need to construct products greater than $56$ when we
deal with $2,4 \pmod{6}$ cases.  Finally, the first row tells us that we
need to construct products greater than $28$ for the $1,5 \pmod{6}$ cases.

\begin{figure}[ht]
\caption{ The shape for computing $\delta(377)$ working modulo $6$.   }
\label{example_377}
    \centering
    \resizebox{.9\textwidth}{!}{

\begin{tikzpicture}

\fill[lightgray] (1,1) -- (29,1) -- (29,13) -- (12, 13) --  (12, 12) -- (11, 12) -- ( 11, 11) -- (10, 11) -- ( 10, 10 ) -- (9, 10) -- (9, 9) -- (8, 9) -- (8,8) -- (7, 8) -- (7,7) 
-- ( 6, 7) -- (6,6) -- (5,6) -- (5,5) -- (4, 5) -- (4,4) -- ( 3, 4) -- (3,3) -- (2,3) -- (2,2) -- ( 1,2) --  cycle; 
\fill[gray] (7,5) -- (7, 6) -- (8, 6) -- (8,5) -- cycle;
\fill[gray] (11,5) -- (11, 6) -- (12, 6) -- (12,5) -- cycle;
\fill[gray] (13,5) -- (13, 6) -- (15, 6) -- (15,5) -- cycle;
\fill[gray] (16,4) -- (16, 6) -- (18, 6) -- (18, 4) -- cycle; 
\fill[gray] (7,7) -- (7, 8) -- (8, 8) -- (8, 7) -- cycle; 
\fill[gray] (8,8) -- (8, 9) -- (9, 9) -- (9, 8) -- cycle; 
\fill[gray] (19,4) -- (21, 4) -- (21, 5) -- ( 22, 5) -- (22, 4) -- (24, 4) -- ( 24, 6) -- (19, 6) -- cycle;
\fill[gray] (25,4) -- (27, 4) -- (27, 5) -- ( 28, 5) -- (28, 4) -- (29, 4) -- ( 29, 6) -- (25, 6) -- cycle;
\fill[gray] (10,7) -- (12,7) -- (12,12) -- (11, 12) -- (11,11) -- (10, 11) -- (10,10) -- (11,10) -- (11,9) -- (10,9) -- cycle;
\fill[gray] (13,7) -- (16, 7) -- (16,8) -- (15,8) -- (15,9) -- (14,9) -- (14,10) -- (13 ,10) -- cycle;
\fill[gray] (13,10) -- (15, 10) -- (15, 12) -- ( 13, 12) -- cycle ;
\fill[gray] (15,9) -- (16,9) -- (16,10) -- (15,10) -- cycle;
\fill[gray] (15,11) --(29, 11) -- (29, 13) -- (15,13) -- cycle;
\fill[gray] ( 16, 10) -- (29, 10) -- (29, 11) -- (16, 11) -- cycle;
\fill[gray] (16,7) -- (18,7) --(18,10) -- (17, 10) -- (17,9) -- (16, 9) -- cycle;
\fill[gray] (19,7) -- (24, 7) -- (24,8) -- (21,8) -- (21, 9) -- (22,9) -- (22,10) -- (19,10) -- cycle; 
\fill[gray] (22,8) -- (29,8) -- (29, 10) -- (22, 10) --cycle;
\fill[gray] (25,7) -- (29, 7) -- (29, 8) -- (25, 8) -- cycle; 

\draw (0, 0) grid (30, 14);
\draw [ultra thick] (1,0) -- (1,14);
\draw [ultra thick] (0,1) -- (30,1);

\foreach \j in {1,...,13} {
	\foreach \i in {1,...,29} {
        		\pgfmathtruncatemacro{\label}{\i*\j};
        		\node (\label) at (\i+ .5, \j +.5) {\label};
}}

\foreach \j in {1,...,13} {
	\foreach \i in {1,...,1} {
        		\pgfmathtruncatemacro{\label}{\i*\j};
        		\node (\label) at (.5, \j +.5) {\label};
}}

\foreach \j in {1,...,1} {
	\foreach \i in {1,...,29} {
        		\pgfmathtruncatemacro{\label}{\i*\j};
        		\node (\label) at (\i + .5, .5) {\label};
}}

\node (x) at (.5,.5) {\large $\times$};

\end{tikzpicture} 

 }
\end{figure}

It is possible to create rules for the evaluation of $\delta(mp)$ via counting
arguments.  We count consecutive products in
residue classes modulo $w = \lcm(1, 2, 3, \ldots, m-1)$.  The third author
created a website~\cite{shapes}
that may be used to count the products constructed and
display the shape associated with a $\delta(n)$ computation when using
Algorithm $\algdeltan$ naively, or 
with a modulus of $w = 1, 2, 6, 12, 60,$ or $120$.

\subsection{Subquadratic Tabulation}

Recall that if all $n^2$ bits of $A$ can be held at once in Algorithm
$\algMn$, then tabulation and evaluation are essentially the same problem.  
We apply this idea to computing $\delta(n)$.  
Consider the use of Algorithm $\algdeltan$ in
computing $\delta(6\cdot7)$, $\delta(6\cdot9)$, $\delta(6\cdot11)$, and
$\delta(6\cdot13)$.  The divisor list for each of these is of the form $[ 1,
6\cdot k], [2, 3 \cdot k], [3, 2\cdot k]$, and $[ 6, k]$ for $k = 7, 9, 11,
13$.  One shape is always a subset of the next shape and so the set of
distinct products in each shape is a subset of the next such set.  Rather than
think of four independent computations, we consider the one computation of
$\delta(6\cdot 13)$.  Unlike in Algorithm $\algdeltan$ where the bit
vector storing distinct products is populated by a row of the multiplication
table, we will populate the bit vector by incrementally shifting the
end-points of the rectangles.  While computing $\delta(6\cdot13)$ we
can ``learn'' $\delta(6\cdot7)$, $\delta(6\cdot9)$, and $\delta(6\cdot11)$.
Instead of computing $\delta(6\cdot9)$ from the beginning,
we use the computation of $\delta(6\cdot 7)$ and only account for the new
products that may arise.

In general, this requires that we tabulate $\delta(n)$, for those $n$ that
have similar shapes.  For a fixed $m$ and primes $p\approx q$, 
the divisor lists of $mp$ and $mq$ are
very similar.
In particular, if both primes are larger than $m$, the first
entries in the divisor lists correspond only to the divisors of $m$.  
If $p < q$, we may
re-use the bit vector from computing $\delta(mp)$ to compute $\delta(mq)$. 
All we need to account for are the
new products that appear
as the corresponding rectangles are shifted.

\SetAlgoRefName{\multipliers}
\begin{algorithm}
\caption{Computing $\delta(mq)$ given $\delta(mp)$ for
primes $p$, $q$ ($m < p < q$).}
\SetKwInOut{Input}{Input} \SetKwInOut{Output}{Output}

\Input{A bit vector  $A$ of length $mq$ with weight $w$ containing the products
from computing $\delta(mp)$.  The divisor lists for $mp$ and $mq$:  $D_p =
[[d_0=1,mp],[d_1, mp/d_1], \ldots ]$ and $D_q =  [[d_0=1,mq],[d_1, mq/d_1],
\ldots ]$ both of length $\ell$.}
\Output{$\delta(mq)$}
\BlankLine
Initialize counters $i = 1$ and $k =0$\\
\For{ $ i < D_p[\ell -1][0]$ }{
\If{ $i == D_p[k][0]$ }{ Increment $k$}
  \For{ $D_p[k][1] \leq j < D_q[k][1]$ }{
  \If{   $A[ij] == 0$ }{ 
  Set $A[ij] = 1$ \\
  Increment $w$ }
  }
 } 
 \Return  $w$
\end{algorithm}

\begin{Lem}			\label{lem:ndq}
If $mq \le n$, then
Algorithm $\multipliers$ computes $\delta(mq)$ in time 
$O(m d(q) \log{n})$, where $d(q):=q-p$.
\end{Lem}

\begin{proof}
There are $O(m\log n)$ individual products to check per unit shift.
\end{proof}

The benefit of Algorithm~$\multipliers$
over Algorithm~$\algdeltan$ is that, 
while computing $\delta(mq)$, we learn $\delta(mp)$ for all $p < q$.  
In computing $M(n)$, we may compute $\delta(mq)$ at a cost of $O(n\log n)$, 
but in the process we learn $\delta(mp)$ for all prime $p$,
$m < p < q$, for no additional cost.

\begin{theorem}\label{fast_comp} 
Algorithms $\algdeltan$ and $\multipliers$ may be combined
to tabulate $M(n)$ in time
$O (n^2/L^{1/\sqrt{2} + o(1)} ),$
where  $L = L(n) := \exp{( \sqrt{ \log n \log \log n }) } $.
\end{theorem}

\begin{proof}
Let $\gamma$ be a real parameter with $0<\gamma<1$, to be chosen
later.  We split the integers $k\le n$ into two classes.  The first
consists of $k$ that are $L^\gamma$-smooth, that is, all prime factors of
$k$ are at most $L^\gamma$.  There are $n/L^{1/(2\gamma)+o(1)}$
such numbers $k$, as $n\to\infty$, see \cite{CEP}.  For these
values of $k$ we compute $\delta(k)$ via Algorithm~$\algdeltan$, accounting
for a run-time of $O(n^2/L^{1/(2\gamma)+o(1)})$.   The second
class consists of those $k$ that are not $L^\gamma$-smooth; 
write such $k$ as $mq$,
where $q$ is the largest prime factor of $k$, so that $q>L^\gamma$.
Since $k\le n$, the pairs $(m,q)$ that can arise here have
$m < n/L^\gamma$.  For each such pair $(m,q)$ take the largest
prime $Q$ with $mQ\le n$ and compute $\delta(mQ)$ using
Algorithm~$\multipliers$,
so learning $\delta(mq)$ for all primes $q\le Q$, and in particular, for
all primes $q$ with $L^\gamma< q\le Q$.  For each $m$ the run-time
is $O(n\log n)$ by Lemma~\ref{lem:ndq}, 
so the total run-time for all such values
of $m$ is $O(n^2/L^{\gamma+o(1)})$.
These two computations are balanced when 
$\gamma = 1/\sqrt{2}$,
proving the theorem.
\end{proof}

For reference we let \emph{Algorithm~$\fastalg$}
be the algorithm defined by the above proof.

\section{Monte Carlo Estimations}\label{approx_eval}

If $n$ is too large for the exact computation of $M(n)$ to be feasible, 
we can resort to Monte Carlo estimation of $M(n)$.
In the following we describe two different Monte Carlo algorithms, 
which we call the \emph{Bernoulli} and \emph{product} algorithms.
In the descriptions of these two algorithms, we assume that $n$ is fixed,
and $p$ denotes a probability (not a prime number).

\subsection{The Bernoulli Algorithm}		\label{subsec:Bernoulli}
		
We perform a sequence of $T \ge 2$ trials, where each trial involves choosing
a random integer $z \in [1,n^2]$. The integers $z$ are
assumed to be independent and uniformly distributed.
For each $z$, we count a \emph{success} if $z$ appears in the
$n \times n$ multiplication table, i.e.\ if $z$ can be written
as $z = xy$, where \hbox{$1 \le x \le y \le n$}.  
Let $S$ be the number of successes after $T$ trials. 
Since we are performing a sequence of $T$
Bernoulli trials with probability of success
$p = M(n)/n^2$, the expected number of successes is 
$\E(S) = pT$, and the variance is 
$\V(S) = \E((S-pT)^2) =  p(1-p)T$.
Thus, an unbiased estimate of $M(n)/n^2$ is given by $\widehat{p} = S/T$,
and the variance of this estimate is $p(1-p)/T$.
For large $T$ the error $M(n)/n^2 - S/T$ is asymptotically
normally distributed. 
By the ``law of the iterated logarithm''~\cite{Khinchin24}, 
this error is almost surely $O((T^{-1}\log\log T)^{1/2})$ as
$T \to \infty$.

\begin{remark}					\label{rem:Bessel}
In a practical computation, $p$ is unknown, but 
an unbiased estimate of the variance of the error is
\hbox{$\widehat{p}(1-\widehat{p})/(T-1)$}, 
where the denominator $T-1$ takes into account 
the loss of one degree of freedom in using the
sample mean $\widehat{p}$ instead of the population 
mean $p$. This is known as \emph{Bessel's correction},
and was used by Gauss~\cite{Gauss-Bessel} as early as 1823.

\end{remark}

\subsection{The Product Algorithm}		\label{subsec:product}

In this algorithm, each trial takes $z = xy$, where $x$ and $y$ are
independently and uniformly distributed integers in $[1,n]$. Thus,
$z$ is guaranteed to appear in the $n \times n$ multiplication table.
Let $\nu = \nu(z) \ge 1$ denote the number of times that $z$ appears 
in the table.
The probability that a trial samples $z$ is $\nu(z)/n^2$.
Thus, $\E(1/\nu) = M(n)/n^2 = p$ (where $p$ is as in the Bernoulli algorithm).
Consider a sequence of $T$ independent trials, giving values
$\nu=\nu_1, \ldots, \nu_T$.
An unbiased estimate of $M(n)/n^2$ is given by
$E := T^{-1}\sum_{1\le j \le T}  1/\nu_j$,
and the variance of this estimate is
$V := T^{-1}\E((\nu^{-1} - p)^2)$.
Lemma~\ref{lemma:compare_variance} shows that, for the same values of $T$
and $n$, the variance in the estimate of $M(n)/n^2$ given by the product
algorithm is no larger than that given by the Bernoulli algorithm.
\begin{Lem}	\label{lemma:compare_variance}
If $V$ is the variance of the estimate $E$ after $T$ trials of the
product algorithm, then $V \le p(1-p)/T$.
\end{Lem}
\begin{proof}
Using $p = \E(\nu^{-1})$, we have
\[V = T^{-1}\E((\nu^{-1}-p)^2)
    = T^{-1}(\E(\nu^{-2}) - p^2).
\]
Since $\nu$ is a positive integer,
$\nu^{-2} \le \nu^{-1}$, and
$\E(\nu^{-2}) \le \E(\nu^{-1}) = p$.  It follows that
$V \le T^{-1}(p - p^2)$, as desired.
\end{proof}
\begin{remark}
It is easy to see that equality holds in
Lemma~$\ref{lemma:compare_variance}$
only in the trivial case $n=1$. From Ford's result~\eqref{eq:FordM}, we have
$TV = O(1/\Phi(n))$ as $n \to \infty$.

\end{remark}

An unbiased estimate of the variance of the error for the product algorithm
in terms of computed quantities is 
$\sum_{1\le j \le T} (\nu_j^{-1}-E)^2/(T(T-1))$, 
see Remark~\ref{rem:Bessel}.

\subsection{Avoiding Factorization via Bach/Kalai}	\label{subsec:Bach}

For the Bernoulli algorithm, we have to determine if an integer $z \in [1,n^2]$
occurs in the $n \times n$ multiplication table. Equivalently, we have to
check if $z$ has a divisor $d$ satisfying $z/n \le d \le n$.
A straightforward algorithm for this would first find the prime power
factorization of~$z$, then attempt to construct a divisor $d$ in the
interval $[z/n, n]$, using products of the prime factors of
$z$.

Similarly, for the product algorithm, we have to count the number of
divisors $d$ of $xy$ in the interval $[xy/n, n]$.
A straightforward algorithm for this would first find the prime power
factorizations of $x$ and $y$.

To avoid having to factor the random integers $z$ (or $x$ and $y$) occurring
in the Bernoulli (or product) algorithms, we can generate random integers
\emph{along with their prime power factorizations}, using the algorithms
of Bach~\cite{bach_factor} or Kalai~\cite{kalai_factor}.
This is much more efficient, on average, than generating random integers and
then attempting to factor them, since the integer factorization problem is
not known to be solvable in polynomial time and is time consuming in
practice for many inputs.

The algorithm described by Bach, specifically
his ``Process R'', returns an integer
$x$ uniformly distributed in the interval $(N/2,N]$,
together with the prime power factorization of $x$. Using Bach's algorithm,
which we call ``procedure $R$'', it is easy to give a recursive
procedure $B$ which returns $x$ uniformly distributed in the
interval $[1,N]$, together with the prime power factorization of $x$.
For details see Algorithm~$\bachmod$. 
The following comments on the complexity of Bach's algorithm also apply to
procedure~$B$.

\SetAlgoRefName{\bachmod}
\begin{algorithm}
\caption{Modification of Bach's algorithm}
{\bf procedure} {$R(N)$}
\BlankLine
\SetKwInOut{Input}{Input} \SetKwInOut{Output}{Output}
\Input{A positive integer $N$}
\Output{A random integer $x\in (N/2,N]$ and its prime power factorization}
\BlankLine
Details omitted: see Bach~\cite[``Process R'', pg.~184]{bach_factor}\\
{\bf end procedure} $R$
\BlankLine
{\bf procedure} {$B(N)$}
\BlankLine
\Input{A positive integer $N$}
\Output{A random integer $x\in [1,N]$ and its prime power factorization}
\BlankLine
\If{ $N == 1$ }{ \Return $1$ }
generate random real $u$ uniformly distributed in $[0,1)$\\
\If{ $u < \lfloor N/2 \rfloor/N$ }{ \Return $B(\lfloor N/2 \rfloor)$ }
\Else{ \Return $R(N)$ }
{\bf end procedure} $B$
\end{algorithm}

The expected running time of Bach's algorithm is dominated by the time for
primality tests.\footnote{More precisely, Bach's algorithm requires prime
power tests, but it is relatively easy to check if an integer is a perfect
power (see Bernstein~\cite{Bernstein-powers}), so primality tests and prime
power tests have (on average) almost the same complexity.  Also, it is
possible to modify Bach's algorithm so that only primality (not prime power)
tests are required. Thus, we ignore the distinction between primality tests
and prime power tests.} Bach's algorithm requires, on average, $O(\log N)$
primality tests. The AKS deterministic primality test~\cite{AKS,LP} requires
$(\log N)^{O(1)}$ bit-operations, so overall Bach's algorithm has
average-time complexity $(\log N)^{O(1)}$. In our implementation, we
replaced the AKS primality test by the Miller--Rabin Monte Carlo
test~\cite{Burthe,LP2,Miller76,Rabin80}, which is much faster,
at the cost of a small probability of error.%
\footnote{The probability of
error can be reduced to $\le 4^{-k}$ by repeating the test $k$
times with independent random inputs,
see \cite{Rabin80}.} 
A small probability of an error (falsely 
claiming that a composite integer is prime) is acceptable when the overall
computation is a Monte Carlo estimation. Such errors will have a negligible
effect on the final result, assuming that the number of trials is large.

Kalai~\cite{kalai_factor} gave an algorithm with the same inputs and outputs
as our modification (procedure~$B$) of Bach's algorithm, but much simpler and
easier to implement. The disadvantage of Kalai's algorithm is that it
is asymptotically slower than Bach's, by a factor of order $\log N$.
More precisely, Kalai's algorithm requires, on average, of order
$(\log N)^2$ primality tests, whereas procedure~$B$ requires of
order $\log N$ prime power tests.  
We implemented both algorithms using Magma~\cite{Magma},
and found that, as expected, Kalai's algorithm was slower than procedure~$B$ 
for $N$ sufficiently large. With our implementations\footnote{Further
details concerning our implementations, and approximations/optimizations
valid for very large $N$, may be found 
in~\cite{multiplication-HK, multiplication-CARMA}.}, 
the crossover
point was  $N \approx 2^{45}$. For $N = 2^{100}$, procedure~$B$ was
faster by a factor of about $2.2$. 

\section{Implementations and Results}	\label{sec:numerical}

We used several independent implementations of Algorithm $\algMn$
(with segmentation), and
three independent implementations of
Algorithm $\algdeltanT$ in three different languages: C, C++, and Sage.%
\footnote{We thank Paul Zimmerman for verifying some of our results
using Sage.}
The published exact computations in \cite{b_k} are of the form
$M(2^n -1)$ for $1 \leq n \leq 17$.  In Table~\ref{tab:BKext}, we include
$18 \leq n \leq 30$.  
The entries in Table~\ref{tab:BKext} were computed independently
using both Algorithm~$\algMn$ and Algorithm~$\algdeltanT$.
No discrepancies were found.\footnote{The entries given
in OEIS A027417 
differ by one because they include the zero product.}
Timing comparisons are difficult as different (time-shared) computer
systems were used, but we estimate that
Algorithm ${\algdeltanT}$ was about three times
faster than Algorithm $\algMn$ for $n = 30$.

\begin{table}[ht]
\begin{center}
\begin{tabular}{ | r | r | r | r | } \hline
  $k$ & $M(2^k - 1)$  & $k$ & $M(2^k - 1)$ \\ \hline
18 & 14081089287 & 25 & 209962593513291\\
19 & 55439171530 &   26 &  830751566970326 \\ 
20 & 218457593222 &   27 & 3288580294256952 \\ 
21& 861617935050  & 28 & 13023772682665848 \\ 
22 &3400917861267 &   29 & 51598848881797343 \\ 
23 &13433148229638 &    30 & 204505763483830092 \\
24 &53092686926154 & &\\ \hline
\end{tabular}
\end{center}
\caption{Extension of the Brent-Kung computation}
\vspace*{-10pt}
\label{tab:BKext}
\end{table}

A table of $M(k\cdot 2^{10})$ for $1 \leq k \leq 2^{20}$ was computed by the
third and fourth authors \cite{m_n_table}.  The computation used a wheel
modulus approach as described \S$\ref{modcomp}$ with $w = 60$.  The
computation took about 7 weeks on Butler University's BigDawg cluster which
has 32 Intel Xeon E5-2630 processors (a total of 192 cores). 
Table~\ref{tab:runtimes} shows the time (in seconds) to compute $\delta(n)$
for all $n \in ( 10^8, 10^8 + 10^3]$ on an Intel i7-4700 with 16GB RAM,
using various values of the modulus~$w$. It can be seen that using larger
moduli provides a significant speedup (at the cost of increased program
complexity).

\begin{table}[ht]
\begin{center}
\begin{tabular}{ | c | r |  } \hline
Algorithm  & time (s)  \\ \hline
Algorithm $\algdeltan$ & 909 \\
$\pmod{1}$ & 302 \\ 
$\pmod{2}$ & 184\\ 
$\pmod{6}$ & 106\\
$\pmod{12}$ & 85 \\
$\pmod{60}$ & 59 \\ \hline
\end{tabular}
\end{center}
\caption{Runtime comparison}
\vspace*{-10pt}
\label{tab:runtimes}
\end{table}

Algorithm $\algMn$, implemented in C, ran on the ARCS computer
system at the University of Newcastle, Australia. The computer nodes used
were a mixture of 2.2 GHz Intel Xeon 3 and 2.6GHz Intel Xeon 4. 

\smallskip

We now consider Monte Carlo algorithms for approximating $M(N)$,
where $N := 2^n-1$.
First consider the case $n=30$, $N=2^{30}-1$, 
for which we know the exact value $M(N) = 204505763483830092$ 
from our deterministic computations.
Taking $T = 10^6$ trials of the ``product'' Monte Carlo algorithm,
we estimate
$M(N)/N^2 = 0.17750$, 
whereas the correct value to 5 decimals (5D) is
$M(N)/N^2 = 0.17738$. 
The variance estimate here is $V = 2.873\times 10^{-8}$,
so $\sigma := V^{1/2} \approx 0.00017$. 
Thus, the Monte Carlo estimate is as accurate as predicted
from the standard deviation $\sigma$.
The same number of trials with the Bernoulli algorithm gives
variance $1.459 \times 10^{-7}$, larger by a factor of about five.
Thus, the product algorithm is more efficient (other things being
equal), as predicted by
Lemma~\ref{lemma:compare_variance}.
In practice the comparison is not so straightforward, because the product
algorithm requires checking more divisors (on average) than the Bernoulli
algorithm, and has a larger space requirement.

The results of some Monte Carlo computations are given to 4D in
Table~\ref{tab:MC}. 
For $n > 10^4$ we used an approximation described
in~\cite{multiplication-HK} to avoid dealing with $n$-bit integers
(essentially by using a logarithmic representation).
We used
the product algorithm (mainly for $n < 10^6$)
and the Bernoulli algorithm (mainly for $n \ge 10^6$),
combined with
Bach's algorithm (described in Section~\ref{approx_eval}).
The Bernoulli algorithm was preferred for $n \ge 10^6$ because of its
smaller space requirements. 
Kalai's algorithm was used for confirmation
(mainly for $n \le 100$).  
 
The second column of Table~\ref{tab:MC} gives an estimate of $M(N)/N^2$,
and the last column gives the normalized value
$(N^2/M(N))/\Phi(N)$. By Ford's result~\eqref{eq:FordM}, this should
be bounded away from $0$ and $\infty$ as $n \to \infty$. 
The third column gives $10^4\sigma$, where $\sigma^2$ is an 
estimate of the variance of the corresponding entry in the second column.
Because of the factor $10^4$, this 
corresponds to units in the last place (ulps) for the second column.
Since the entries
in the third column are bounded by $0.12$,
the entries in the second column are unlikely to be in error by
more than $0.7$ ulp.
Similarly, the entries in the last column of the table
are unlikely to be in error by more than 1~ulp.\footnote{%
Table~\ref{tab:MC} is extended to $n = 5\times 10^8$ (but with lower
accuracy) in~\cite{multiplication-HK,multiplication-CARMA}.}
The entries for $n \le 30$ may be verified (up to the predicted accuracy)
using the exact results of Table~\ref{tab:BKext}.

\begin{table}[ht]
\begin{center}
\begin{tabular}{ | c | c | c| c |c | } \hline
  $n$ & $M(N)/N^2$  		& $10^4 \sigma$	& \underbar{trials}
      & \underbar{$N^2/M(N)$} \\ 
      & ($N = 2^n-1$)	&& $10^8$ & $\Phi(N)$ \\ \hline
20 & 0.1987     & 0.12 & $2$ & 0.9414\\ 
30 & 0.1774    	& 0.02 & $100$ & 0.8213\\
40 & 0.1644	& 0.02 & $100$ & 0.7549\\
50 & 0.1552	& 0.02 & $100$ & 0.7112\\
$10^2$ & 0.1311 & 0.02 & $100$ & 0.6068\\
$10^3$ & 0.0798	& 0.02 & $100$ & 0.4264\\
$10^4$ & 0.0517	& 0.01 & $100$ & 0.3435\\
$10^5$ & 0.0348	& 0.06 & $2$ & 0.2958\\
$10^6$ & 0.0240	& 0.05 & $10$ & 0.2652\\
$10^7$ & 0.0170 & 0.05 & $6.7$ & 0.2432 \\
$10^8$ & 0.0121 & 0.10 & $1.32$ & 0.227 \\
\hline
\end{tabular}
\end{center}
\caption{Monte Carlo computations}
\label{tab:MC}
\vspace*{-10pt}
\end{table}

It has not been shown that the numbers in the last column of
Table~\ref{tab:MC}
should tend to a limit as $N\to\infty$.  Ford's result (2) shows that
the $\limsup$ and $\liminf$ are finite and positive,
but not that a limit exists.
A non-rigorous extrapolation of our experimental results,
described in more detail
in ~\cite{multiplication-HK, multiplication-CARMA},
suggests that the limit (if it exists) is about $0.12$.
Clearly convergence is very slow. Perhaps this is
to be expected, given that $\Phi(n)$ grows very slowly.

In some similar problems the corresponding limit does not exist.
For example,  
let $S(x)$ be the number of \hbox{$n \leq x$} 
such that the number of divisors of $n$ is at least $\log x$.  
Norton~\cite{norton} showed that there are positive constants $c_1, c_2$ with 
\hbox{$c_1 < R(x) < c_2$} for $x$ sufficiently large,
where $R(x) = S(x)x^{-1}(\log x)^c(\log \log x)^{1/2}$.  
Later, Balazard, \emph{et~al.}~\cite{bnpt} showed that 
$\lim_{x \to \infty} R(x)$ does not exist.
Thus, it would not be too surprising if
$\lim_{N\to\infty} N^2/(M(N)\Phi(N))$ failed to exist.
However, we have not detected any numerical evidence for oscillations in
the last column of Table~\ref{tab:MC}, so we would expect the
$\liminf$ and $\limsup$ to be close, even if unequal.

\pagebreak[3]

\end{document}